\documentclass{amsart}
\usepackage{amssymb,amsfonts,amsmath}
\usepackage[all]{xy}
\usepackage[shortlabels]{enumitem}
\usepackage{hyperref, color}
\usepackage{commath}
\usepackage{esdiff}	
\usepackage{graphicx,subfigure}	
\usepackage[center]{caption}
\usepackage{hyperref}

\newcommand{\C}{\mathbb C}
\newcommand{\R}{\mathbb R}

\newcommand{\h}{\mathbb H}
\newcommand{\s}{\mathbb S}
\newcommand{\im}[1]{\mathrm{Im}#1}
\newcommand{\re}[1]{\mathrm{Re}#1}
\newcommand{\iso}[1]{\mathrm{Iso}(#1)}

\DeclareMathOperator{\id}{\mathrm{Id}}

\DeclareMathOperator{\psl2}{\widetilde{\mathrm{PSL}}_{2}(\R)}

\DeclareMathOperator{\S2R}{\mathbb{S}^2\times\mathbb{R}}

\newtheorem{theorem}{Theorem}[section]
\newtheorem{lemma}[theorem]{Lemma}
\newtheorem{proposition}[theorem]{Proposition}
\newtheorem{corollary}[theorem]{Corollary}

\newtheorem*{th-A}{Theorem \ref{th-A}}
\newtheorem*{prop-cylinder}{Proposition \ref{cylinders}}
\newtheorem*{cor-g-harmonic}{Corollary \ref{g-harmonic}}
\newtheorem*{prop-antiholo}{Proposition \ref{anti-holomorphic}}
\newtheorem*{lemma-curvature}{Lemma \ref{curvatura}}

\theoremstyle{definition}
\newtheorem{definition}[theorem]{Definition}
\newtheorem{example}[theorem]{Example}
\newtheorem{remark}[theorem]{Remark}

\numberwithin{equation}{section}
\numberwithin{figure}{section}

\begin{document}

\title[The Gauss map of minimal surfaces in $\s^2 \times \R$]{The Gauss map of minimal surfaces in $\s^2 \times \R$}

\author[Iury Domingos]{Iury Domingos$^{1,2}$}
\address{$^1$ Universit\'e de Lorraine \\ CNRS, IECL\\
F-54000 Nancy, France.}
\address{$^2$ Universidade Federal de Alagoas \\ Instituto de Matem\'{a}tica  \\ Campus A. C. Sim\~{o}es, BR 104 - Norte, Km 97, 57072-970.
Macei\'o - AL, Brazil.}
%
%
%
%
\email{domingos1@univ-lorraine.fr}
%
%
%
\thanks{
This work was partially conducted during a scholarship supported by the International Cooperation Program CAPES/COFECUB Foundation at the University of Lorraine; financed by CAPES -- Brazilian Federal Agency for Support and Evaluation of Graduate Education within the Ministry of Education of Brazil.}

\begin{abstract}
In this work, we consider the model of $\S2R$ isometric to $\R^3\setminus \{0\}$,\break endowed with a metric conformally equivalent to the Euclidean metric of $\R^3$,
and we define a Gauss map for surfaces in this model likewise in the $3-$Euclidean space. We show as a main result that any two minimal conformal immersions in $\S2R$ with the same non-constant Gauss map differ by only two types of ambient isometries: either $f=(\id,T)$, where $T$ is a translation on $\R$, or $f=(\mathcal{A},T)$, where $\mathcal{A}$ denotes the antipodal map on $\s^2$. Moreover, if the Gauss map is singular, we show that it is necessarily constant, and then only vertical cylinders over geodesics of $\s^2$ in $\S2R$ appear with this assumption. We also study some particular cases, among them we prove that there is no minimal conformal immersion in $\S2R$ which the Gauss map is a non-constant anti-holomorphic map.
\end{abstract}

\maketitle

\section{Introduction}

A fundamental contribution to the theory of minimal surfaces in $\R^3$ is due to K. Weierstrass: Given $g:\Omega\subset\C\to\bar{\C}$ a meromorphic function and $\eta$ a holomorphic $1-$form defined on an open set $\Omega\subset\C$, such that whenever that $g$ has a pole of order $m$ at $z\in\Omega$ implies that $\eta$ has a zero of order $2m$ at $z\in\Omega$, the Weierstrass Representation states that
\[X(z)=X(z_0)+\re{\int_{z_0}^{z}\Big(\frac{1}{2}(1-g^2)\eta,\frac{i}{2}(1+g^2)\eta,g\eta\Big)}\]
define a conformal minimal surface in $\R^3$ with Gauss map $g$. In light of this, the relationship between the geometry of minimal surfaces and its Gauss map $g$ has been studied in the last century, and some Weierstrass-type Representations has been established for minimal surfaces in others $3-$dimensional ambient Riemannian manifolds.

In general, given an isometric (or conformal) immersion into some $3-$dimensional ambient Riemannian manifold, the first problem is to define a Gauss map of a such surface into the unit $2-$sphere $\s^2$, preserving some proprieties such as the Gauss map defined in $\R^3$, whenever possible. The second problem is how ``good'' is this definition, that is, if the Gauss map defined carries important informations about the surface. 
However, the harmonicity condition of this map is not satisfied generally for an arbitrary ambient manifold.

For simply connected $3-$homogeneous manifolds with a $4-$dimensional isometry group, i.e., a Riemannian fibration of bundle curvature $\tau$ over a $2-$space form $\mathbb{M}^2_\kappa $ with sectional curvature $\kappa$, satisfying $\kappa-4\tau^2\neq 0$, the topic of Gauss map theory is very active and several results were established in the last decade. For instance, by B. Daniel \cite{Daniel11} on the Heisenberg Group $\mathrm{Nil}_3$, by B. Daniel, I. Fernandéz and P. Mira \cite{DanielFernandezMira15} on $\psl2$, by I. Fernandéz and P. Mira \cite{FernandezMira07} on $\h^2\times\R$ and by M. Leite and J. Ripoll \cite{LeiteRipoll11} on $\S2R$ and $\h^2\times\R$.

W. Meeks and J. Pérez studied surfaces with constant mean curvature in $3-$
dimensional simply connected metric Lie groups in the survey \cite{MeeksPerez12}. Due to the Lie group structure, they presented a \emph{left invariant Gauss map} for an oriented surface in metric Lie groups. This map assumes values in the unit
$2-$sphere of the metric Lie algebra and it can be generalised to higher dimensions. In \cite{MeeksMiraPerezRos13}, W. Meeks, P. Mira, J. Pérez and A. Ros found a $2-$order equation satisfied by the \emph{left invariant Gauss Map}. They established, from this equation, a Weierstrass-type representation for surfaces with constant mean curvature in a $3-$dimensional metric Lie groups. However, since $\s^2\times\R$ is the only simply connected $3-$homogeneous Riemannian manifold that does not carry a structure of $3-$dimensional simply connected Lie group, with a left invariant metric, this ambient manifold was excluded from this list.

Motivated by the exceptional case of $\S2R$ not discussed in \cite{MeeksMiraPerezRos13,MeeksPerez12}, we define a new Gauss map for surfaces in $\S2R$. We use the model of $\S2R$ isometric to $\R^3\setminus \{0\}$, endowed with a metric conformally equivalent to the Euclidean metric of $\R^3$. Therefore, given a surface in $\S2R$ we define the Gauss map likewise in $\R^3$, identifying each unit $2-$sphere on each tangent plane to this surfaces with the unit $2-$sphere $\s^2\subset\R^3$.

The goal of this paper is to study minimal surfaces in $\S2R$ from this Gauss map.
More  specifically,  we  are  interested  to  know  when  a  such  minimal  immersion  is determined by its conformal structure and its Gauss map. As a main result, we prove that:

\begin{th-A}
Let $X:\Sigma\to\S2R$ be a minimal conformal immersion and $g$ be its non-constant Gauss map. If $\hat{X}:\Sigma\to\S2R$ is another minimal conformal immersion with the same Gauss map of $X$, then $\hat{X}=f\circ X$, with $f\in\iso{\S2R}$ given either by $f=(\id,T)$ or $f=(\mathcal{A},T)$, where $\mathcal{A}$ denotes the antipodal map on $\s^2$ and $T$ is a translation on $\R$.
\end{th-A}

Unlike what happens in $\R^3$, any deformation of a minimal conformal immersion preserving this Gauss map is in fact rigid, i.e., it is given by an isometry of $\S2R$. However, we show that the zeroes of the curvature $K$ of the metric induced by $X$ are the singular points of $g$, hence either $K=0$ or its zeroes are isolated, likewise what happens in $\R^3$.

Under hypothesis of anti-holomorphic non-constant Gauss map $g$, we prove that this condition is an obstruction to the existence of a minimal conformal immersion into $\S2R$. For instance:



\begin{prop-antiholo}
There is no minimal conformal immersion $X:\Sigma\to\S2R$ which its Gauss map $g$ is an anti-holomorphic non-constant map.
\end{prop-antiholo}



In the case of minimal conformal immersion into $\S2R$ with a singular Gauss map, we show that only vertical cylinders over geodesics of $\s^2$ in $\S2R$ appear, and then the Gauss map is necessarily constant.

\begin{prop-cylinder}
Let $X:\Sigma\to\S2R$ be a minimal conformal immersion and $g$ be its Gauss map. If $g$ is a singular map then $X(\Sigma)$ is part of a vertical cylinder over a geodesic of $\s^2$ in $\s^2\times\R$. In particular, $g$ is constant.
\end{prop-cylinder}

It is worthwhile to point out our definition of Gauss map does not coincide with the \emph{Twisted normal map} defined by M. Leite and J. Ripoll in \cite{LeiteRipoll11}. In fact, the \emph{Twisted normal map} assumes values into the unit $3-$sphere $\s^3$ and the harmonicity of this map is equivalent to affirm that the surface has constant mean curvature in $\s^2\times\R$. In our case, only totally geodesic $2-$spheres in $\s^2\times\R$ have non-constant harmonic Gauss map.

\begin{cor-g-harmonic}
Let $X:\Sigma\to\S2R$ be a minimal conformal immersion and $g$ be its Gauss map. If $g$ is a non-constant harmonic map then $X(\Sigma)$ is part of a totally geodesic $2-$sphere in $\s^2\times\R$.
\end{cor-g-harmonic}


\section{The model of \texorpdfstring{$\S2R$}{S2xR}}
Let $\s^2$ be the unit $2-$sphere in $\R^3$ endowed with the induced metric, and consider the Riemannian product manifold $\S2R$ endowed with the product metric (referred to as the standard model of $\S2R$). We denote by $\Pi_1:\S2R\to\s^2$ and $\Pi_2:\S2R\to\R$ the projections into the factors $\s^2$ and $\R$, respectively, given by $\Pi_1(y,t)=y$ and $\Pi_2(y,t)=t$.

We consider $\R^3\setminus\{0\}$ endowed with the metric
\[\dif \mu^2=\frac{1}{x_1^2+x_2^2+x_3^2}(\dif x_1^2+\dif x_2^2+\dif x_3^2)\]
and the smooth map $\phi:\S2R\to\R^3\setminus\{0\}$, defined by
$\phi(y,t)=e^ty$, with inverse map given by $\phi^{-1}(x)=(x/|x|,\log|x|)$, for $x=(x_1,x_2,x_3)\in\R^3\setminus\{0\}$. In addition, we define $\pi_j=\Pi_j\circ\phi^{-1}$ for $j=1,2$, i.e., 
\[x  \mapsto \pi_1(x)=\frac{x}{|x|} \text{ \ and \ } x\mapsto \pi_2(x)=\log|x|.\]
A straightforward computation shows that $\phi$ is a global isometry between $\S2R$ and $\R^3\setminus\{0\}$, therefore, we may consider $\S2R$ as $(\R^3\setminus\{0\},\dif \mu^2)$.

Let $\{\partial_{x_j}\}_{j=1,2,3}$ be the canonical orthonormal frame of $\R^3$ and we set $E_j(x)=|x|\partial_{x_j}$, for $x\in\S2R$ and $j=1,2,3$. The frame $\{E_j(x)\}_{j=1,2,3}$ is an orthonormal frame of $\S2R$. If $\bar{\nabla}$ denotes the Riemannian connection of $\S2R$,  by a standard computation we get
\[\everymath={\displaystyle}\begin{array}{ccc}
  \bar{\nabla}_{E_1}E_1=\frac{1}{|x|}(x_2E_2+x_3E_3), & \bar{\nabla}_{E_2}E_1=-\frac{x_1}{|x|}E_2, & \bar{\nabla}_{E_3}E_1=-\frac{x_1}{|x|}E_3,
\end{array}\]

\[\everymath={\displaystyle}\begin{array}{ccc}
  \bar{\nabla}_{E_1}E_2=-\frac{x_2}{|x|}E_1, & \bar{\nabla}_{E_2}E_2=\frac{1}{|x|}(x_1E_1+x_3E_3), & \bar{\nabla}_{E_3}E_2=-\frac{x_2}{|x|}E_3,
\end{array}\]

\[\everymath={\displaystyle}\begin{array}{ccc}
  \bar{\nabla}_{E_1}E_3=-\frac{x_3}{|x|}E_1, & \bar{\nabla}_{E_2}E_3=-\frac{x_3}{|x|}E_2, & \bar{\nabla}_{E_3}E_3=\frac{1}{|x|}(x_1E_1+x_2E_2).
\end{array}\]

Given a vector $V=v_1\partial_{x_1}+v_2\partial_{x_2}+v_3\partial_{x_3}$ tangent at the point $x=(x_1,x_2,x_3)$ of $\S2R$, we will use brackets to express the coordinates of $V$ in the frame $\{E_j(x)\}_{j=1,2,3}$, that is,
\[v_1\partial_{x_1}+v_2\partial_{x_2}+v_3\partial_{x_3}=\frac{1}{|x|}
                                          \begin{bmatrix}
                                            v_1 \\
                                            v_2 \\
                                            v_3 \\
                                          \end{bmatrix}.\]

The isometry group of $\S2R$ is a $4-$dimensional group isomorphic to $\iso{\s^2}\times\iso{\R}$. In particular, it has $4$ connected components: a given isometry can either preserves or reverses the orientation in each factor $\s^2$ and $\R$. Consider $f=(M,T)\in\iso{\s^2\times\R}$ an ambient isometry. If $f$ preserves the orientation of $\R$, i.e., when $T$ is given by $T(t)=s+t$, then $\phi^{-1}\circ f\circ\phi$ is the correspondent isometry of $f$ in our model of $\s^2\times\R$, given by
\[x\in\R^3\setminus\{0\}\mapsto e^{s}M(x),\]
with $s\in\R$ and $M\in O(3)$, where $O(3)$ denotes the $3-$dimensional Orthogonal Group. If $f$ reverses the orientation of $\R$, i.e., when $T$ is given by $T(t)=s-t$, then the correspondent isometry of $f$ in our model of $\s^2\times\R$ is given by
\[x\in\R^3\setminus\{0\}\mapsto \frac{e^{s}}{|x|^2} M(x),\]
with $s\in\R$ and $M\in O(3)$. In order to simplify our notation, given an isometry $f\in\iso{\s^2\times\R}$, we identify $f$ and its correspondent $\phi^{-1}\circ f\circ\phi$ in our model of $\S2R$.

We say that an isometry $f\in\iso{\S2R}$ is a vertical translation of $\S2R$ when it is the identity on the factor $\s^2$ and it preserves the orientation of $\R$, that is, $f=(\id,T)$ where $T(t)=s+t$ is a translation on $\R$, with $s\in\R$ fixed.

\section{The Gauss map}

In this section, we identify
\[\S2R \text{ \ with \ }(\R^3\setminus\{0\},\dif \mu^2).\]

In our model of $\S2R$, we consider $\Pi:\S2R\to\C$ the projection into $\R^2$ (identified with $\C$) and $\pi:\S2R\to\R$ the projection into $\R$, respectively, given by  $\Pi(x_1,x_2,x_3)=x_1+ix_2$ and $\pi(x_1,x_2,x_3)=x_3$.

Let $\Sigma$ be an oriented Riemmanian surface and $z=u+iv$ be a conformal coordinate on $\Sigma$ and consider $X:\Sigma\to\S2R$ a conformal immersion. We define the projections of $X$ on $\C$ and $\R$ by $F=\Pi\circ X: \Sigma\to\C$ and $h=\pi\circ X:\Sigma\to\R$, and then the conformal immersion $X$ is written as $X=(F,h)$.
We denote by $N:\Sigma\to \textrm{U}(\S2R)$ the unit normal field to $\Sigma$, where $\textrm{U}(\S2R)$ is the unit tangent bundle to $\S2R$.

Since $\{E_j(x)\}_{j=1,2,3}$ is a global orthonormal frame of  $\S2R$, then for each $x\in\S2R$, we identify the unit $2-$sphere of $T_{x}(\S2R)$  with the unit $2-$sphere $\s^2\subset \R^3$, i.e., $v_1E_1+v_2E_2+v_3E_3 \in\s^2\subset T_x(\S2R)$
is identified with $v_1\partial_{x_1}+v_2\partial_{x_2}+v_3\partial_{x_3}\in\s^2\subset\R^3$.
Therefore, we may consider the unit normal vector $N$ into the $2-$sphere $\s^2\subset\R^3$ up to this identification.

\begin{definition}The Gauss map of $X$ is the map $g=\varphi\circ N:\Sigma\to\bar{\C}$, where $\varphi: \s^2\to\bar{\C}$ is the stereographic projection with respect to the southern pole, that is, if $N=N_1E_1+N_2E_2+N_3E_3$ then
\[g=\frac{N_1+iN_2}{1+N_3}:\Sigma\to\bar{\C}\]
satisfying
\[N=\frac{1}{1+|g|^2}\begin{bmatrix}
                       g+\bar{g}\\
                       i(\bar{g}-g)\\
                       1-|g|^2
                     \end{bmatrix}.\]
\end{definition}

\begin{remark}
For the other choice of $N$, we replace $g$ by $\widetilde{g}=-1/\bar{g}$.
\end{remark}

\begin{remark}\label{isometries}
Let $f=(M,T)\in\iso{\S2R}$, where $T$ is a translation by $s$ on $\R$, and consider the conformal immersion $\hat{X}=f\circ X$. Then $\hat{N}=(\det M)M\circ N$. In fact, since $f:\S2R\to\S2R$, given by $f(x)=e^sM(x)$, is a linear map then, for $V\in \mathrm{T}_x(\S2R)$, we have $\dif f_x (V) = e^s M(V)$ and the unit normal $\hat{N}$ is given by
\[\hat{N} = \frac{\dif f_{X(z)}(X_u)\times\dif f_{X(z)}(X_v)}{\|\dif f_{X(z)}(X_u)\times\dif f_{X(z)}(X_v)\|}. \]
However,
\[\dif f_{X(z)}(X_u)\times\dif f_{X(z)}(X_v) = e^{2s}(\det M)M(X_u\times X_v)\]
and
\[\|\dif f_{X(z)}(X_u)\times\dif f_{X(z)}(X_v)\| = e^{2s}\|X_u\times X_v\|,\]
once $M\in O(3)$, thus $\hat{N}=(\det M)M\circ N$. Moreover, since $(\det M)M\in SO(3)$, there is $R_M\in SU(2)$ such that $(\det{M})M(z_1)=z_2$ implies $R_M(\varphi(z_1))=\varphi(z_2)$, for $z_1,z_2\in \s^2$,  where $SU(2)$ denotes the Special Unitary Group of degree $2$. Consequently, we get that $\hat{g}=R_M\circ g$. In particular, when $M$ is the antipodal map $\mathcal{A}$, then $\hat{N}=N$ and hence $\hat{g}=g$.
\end{remark}

In the following computations, we adapt the method used by Daniel in \cite{Daniel11} to study the Gauss map of surfaces in the Heisenberg group $\mathrm{Nil}_3$. For sake of clarity, we maintain the Daniel's notations.

We consider the complex function $\eta = 2h_z$. For a conformal immersion $X$, we have
\[X_z=\frac{1}{2|X|}\begin{bmatrix}
    (F+\bar{F})_z \\
    i(\bar{F}-F)_z \\
    \eta
  \end{bmatrix}
  \text{ \ and \ }
  X_{\bar{z}}=\frac{1}{2|X|}\begin{bmatrix}
    (F+\bar{F})_{\bar{z}} \\
    i(\bar{F}-F)_{\bar{z}} \\
    \bar{\eta}
  \end{bmatrix},\]
where $|X|=\sqrt{|F|^2+h^2}$. We note that $X$ is a conformal immersion if, and only if, $|X_z|=0$, that is,
\begin{equation}\label{eq-conformality}
F_z\bar{F}_z=-\frac{\eta^2}{4}.
\end{equation}

We compute
\[X_z\times X_{\bar{z}}=\frac{i}{2|X|^2}\begin{bmatrix}
                                          \re(\eta F_{\bar{z}}-\bar{\eta}F_z) \\
                                          \im(\eta F_{\bar{z}}-\bar{\eta}F_z) \\
                                          |F_z|^2-|F_{\bar{z}}|^2
                                        \end{bmatrix}\]
and we get
\[X_u\times X_v=-2iX_z\times X_{\bar{z}}=\frac{1}{|X|^2}\begin{bmatrix}
                                          \re(\eta F_{\bar{z}}-\bar{\eta}F_z) \\
                                          \im(\eta F_{\bar{z}}-\bar{\eta}F_z) \\
                                          |F_z|^2-|F_{\bar{z}}|^2
                                        \end{bmatrix}.\]
                                        
We note that $g=0$ if, and only if, $N$ is the northern pole, equivalently
\begin{align*}
  \eta F_{\bar{z}}-\bar{\eta}F_z &= 0, \\
  |F_z|^2-|F_{\bar{z}}|^2 &> 0.
\end{align*}
Then by equation \eqref{eq-conformality} we get $\bar{\eta}(|F_z|^2+|\eta|^2/4)=0$ that implies $\eta =0$, since $|F_z|^2>0$. Again, by \eqref{eq-conformality} we get $F_z F_{\bar{z}}=0$, that is, $F_{\bar{z}}=0$. Analogously, $g=\infty$ if, and only if, $N$ is the southern pole, equivalently, $\eta =0$ and $F_z=0$. Therefore $g=0$ or $\infty$ if, and only if, $\eta = 0$. Moreover, $g=0$ if, and only if, $F_{\bar{z}}=0$ and $g=\infty$ if, and only if, $F_z=0$.

We compute
\[\|X_u\times X_v\|=2\langle X_z,X_{\bar{z}}\rangle=\frac{1}{|X|^2}\Big(|F_z|^2+|F_{\bar{z}}|^2+\frac{1}{2}|\eta|^2\Big).\]

Then, when $g\neq \infty$, we get
\[\frac{2g}{1+|g|^2}=\frac{\eta F_{\bar{z}}-\bar{\eta}F_z}{|F_z|^2+|F_{\bar{z}}|^2+\frac{1}{2}|\eta|^2}\text{ \ and \ }
\frac{1-|g|^2}{1+|g|^2}=\frac{|F_z|^2-|F_{\bar{z}}|^2}{|F_z|^2+|F_{\bar{z}}|^2+\frac{1}{2}|\eta|^2},\]
therefore
\[g=\frac{\eta F_{\bar{z}}-\bar{\eta}F_z}{2|F_z|^2+\frac{1}{2}|\eta|^2}.\]

Using \eqref{eq-conformality} in the expression above, we get
\begin{equation}\label{eq-F}
\bar{g}F_z=-\frac{\eta}{2}\text{ \ and \ }F_{\bar{z}} = \frac{g\bar{\eta}}{2},
\end{equation}
which implies that the functions $\eta/\bar{g}$ and $g\bar{\eta}$ can be extended smoothly to the points where $g=0$ or $\infty$. Furthermore, at the points where $g\neq0$ or $\infty$, the metric induced by $X$ is given by
\begin{align}\label{eq-inducedmetric1}
\dif s^2 &= \frac{(1+|g|^2)^2|\eta|^2}{4|g|^2|X|^2}|\dif z|^2.
\end{align}

\section{Minimal immersions into \texorpdfstring{$\S2R$}{S2xR}}

We devote this section to study minimal surfaces in $\S2R$ from the Gauss map defined in the previous section. We recall that, unless otherwise stated, we identify
\[\S2R \text{ \ with \ }(\R^3\setminus\{0\},\dif \mu^2).\]

From now on, we suppose that $X:\Sigma\to\S2R$ is a minimal conformal immersion. As a first result, we translate the minimality condition of $X$ in terms of its Gauss map $g$ and its complex function $\eta$.

\begin{lemma}\label{lemma-g-eta}
Let $X:\Sigma\to\S2R$ be a minimal conformal immersion. Then the Gauss map $g$ and the complex function $\eta$ satisfy
\begin{equation}\label{eta}
\frac{1-|g|^2}{1+|g|^2}(\log \bar{g})_{\bar{z}}=(\log \eta)_{\bar{z}}
\end{equation}
when $g\neq \infty$.
\end{lemma}

\begin{proof}
We consider $U\subset\Sigma$ an open set which $g\neq 0$. Firstly, note that by equations \eqref{eq-F}, we can write
\[X_z=\frac{\eta}{4|X|}\begin{bmatrix}\everymath={\displaystyle}
    \bar{g}-\bar{g}^{-1} \\
    i(\bar{g}+\bar{g}^{-1}) \\
    2
  \end{bmatrix}
  \text{ \ and \ }
  X_{\bar{z}}=\frac{\bar{\eta}}{4|X|}\begin{bmatrix}\everymath={\displaystyle}
    g-g^{-1} \\
    i(g+g^{-1}) \\
    2
  \end{bmatrix}.\]
Therefore, setting
\[A_1 = \frac{\eta}{4|X|}\Big(\bar{g}-\frac{1}{\bar{g}}\Big), \ A_2 = \frac{\eta}{4|X|}\Big(\bar{g}+\frac{1}{\bar{g}}\Big)
 \ \text{ and } \ A_3=\frac{\eta}{2|X|},\]
we have that $\displaystyle X_z = \sum_{j=1}^{3} A_jE_j$ and $\displaystyle X_{\bar{z}} = \sum_{j=1}^{3}\bar{A}_j E_j$.

Given $X:\Sigma\to\S2R$ a conformal immersion, a necessary and sufficient condition for $X$ to be minimal is
\[\bar{\nabla}_{X_{\bar{z}}}X_z=0,\]
that is,
\begin{equation}\label{minimality-condition}
\sum_{k=1}^{3}(A_k)_{\bar{z}}E_k+\sum_{\ell,j=1}^{3} \bar{A}_{\ell} A_j \bar{\nabla}_{E_{\ell}}E_j = 0.
\end{equation}

A straightforward computation gives:
\begin{multline*}
(A_1)_{\bar{z}} = \frac{1}{16|X|^3|g|^2\bar{g}}\Big\{4|X|^2|g|^2(\bar{g}^2-1)\eta_{\bar{z}}+4|X|^2(\bar{g}^2+1)g\eta\bar{g}_{\bar{z}} \\
  -|\eta|^2(\bar{g}^2-1)(|g|^2g\bar{F}-\bar{g}F+2h|g|^2)\Big\},
\end{multline*}
\begin{multline*}
(A_2)_{\bar{z}} = \frac{i}{16|X|^3|g|^2\bar{g}}\Big\{4|X|^2|g|^2(\bar{g}^2+1)\eta_{\bar{z}}+4|X|^2(\bar{g}^2-1)g\eta\bar{g}_{\bar{z}} \\
  -|\eta|^2(\bar{g}^2+1)(|g|^2g\bar{F}-\bar{g}F+2h|g|^2)\Big\}
\end{multline*}
and
\begin{equation*}
(A_3)_{\bar{z}} = \frac{1}{8|X|^3g}\Big\{4|X|^2g\eta_{\bar{z}}-|\eta|^2(g^2\bar{F}-F+2hg)\Big\}.
\end{equation*}
Thus
\begin{multline*}
\sum_{j,k} \bar{A}_j A_k \bar{\nabla}_{E_j}E_k =\\ =\frac{|\eta|^2}{16|X|^3|g|^2}\Big\{(|g|^4+g^2)\bar{F}+(\bar{g}^2+1)F-2h\bar{g}(g^2-1)+4|g|^2\re{F}\Big\}E_1 \\
+\frac{|\eta|^2}{16|X|^3|g|^2}\Big\{i(|g|^4-g^2)\bar{F}+i(\bar{g}^2-1)F+2ih\bar{g}(g^2+1)+4|g|^2\im{F}\Big\}E_2\\
+\frac{|\eta|^2}{8|X|^3|g|^2}\Big\{(1+|g|^4)h-|g|^2\bar{g}F+g\bar{F}\Big\}E_3.
\end{multline*}

The vertical part of \eqref{minimality-condition} is equal to zero if, and only if,
\[4|X|^2|g|^2\eta_{\bar{z}}-|\eta|^2(|g|^2g\bar{F}-\bar{g}F+2|g|^2h)+|\eta|^2((1+|g|^4)h-|g|^2\bar{g}F+g\bar{F})=0,\]
i.e.,
\begin{equation*}
  4|X|^2|g|^2\bar{\eta}_{z}=-|\eta|^2(1-|g|^2)\Omega,
\end{equation*}
where $\Omega = F\bar{g}+\bar{F}g+h(1-|g|^2)\in \R$.

The horizontal part of \eqref{minimality-condition} is equal to zero if, and only if,
\begin{multline*}
-4|X|^2|g|^2\eta_{\bar{z}}+4|X|^2\eta g\bar{g}_{\bar{z}}+|\eta|^2(|g|^2g\bar{F}-\bar{g}F+2h|g|^2) \\
  +|\eta|^2(|g|^2g\bar{F}+\bar{g}F-2|g|^4h+2|g|^2\bar{g}F)=0,
\end{multline*}
i.e.,
\[4|X|^2\eta g\bar{g}_{\bar{z}}-4|X|^2|g|^2\eta_{\bar{z}}+2|\eta|^2|g|^2\Omega=0.\]
Using \eqref{eqminimality-1} in the equation above we get
\begin{equation*}
  4|X|^2g\bar{g}_{\bar{z}} = -\bar{\eta}(1+|g|^2)\Omega.
\end{equation*}

By continuity, equations \eqref{eqminimality-1} and \eqref{eqminimality-2} also hold on a neighbourhood of a point which $g=0$. Therefore,
\begin{align}
4|X|^2g\bar{g}_{\bar{z}}
&= -\bar{\eta}(1+|g|^2)\Omega \label{eqminimality-2},\\
4|X|^2|g|^2\bar{\eta}_{z} &= -|\eta|^2(1-|g|^2)\Omega\label{eqminimality-1}
\end{align}
are necessary and sufficient conditions for $X$ to be minimal, when $g\neq \infty$.

Suppose that $X$ is a minimal conformal immersion. At a point which $g\neq 0$, we have that $\eta \neq 0$, then by \eqref{eqminimality-2} and \eqref{eqminimality-1} we have
\[-\frac{\bar{\eta}\Omega}{4|X|^2g} = \frac{\bar{g}_{\bar{z}}}{1+|g|^2}\text{ \ and \ }-\frac{\bar{\eta}(1-|g|^2)\Omega}{4|X|^2g}=\frac{\bar{g}\bar{\eta}_z}{\eta}.\]
Multiplying the first equation by $(1-|g|^2)$ and substituting it in the second one, we get our assertion. Again by continuity, \eqref{lemma-g-eta} holds at a point where $g=0$.
\end{proof}

\begin{example}[Horizontal surfaces and vertical cylinders]
The simplest examples of minimal surfaces in $\s^2\times\R$ are the horizontal surfaces and the vertical cylinders over geodesics of $\s^2$.

The totally geodesic $2-$spheres $\s^2\times\{a\}$ in the standard model of $\s^2\times\R$, for $a\in\R$, are $2-$spheres of radius $e^a$ centered at the origin of $\R^3$ in our model of $\S2R$.

The vertical cylinders $\gamma\times\R$ over geodesics of $\s^2$ in the standard model of $\s^2\times\R$ are planes passing through the origin of $\R^3$ in our model of $\S2R$. 
\end{example}

\begin{proposition}\label{cylinders2}
Let $X:\Sigma\to\S2R$ be a minimal conformal immersion and $g$ be its Gauss map. Then $g$ is constant if, and only if, $X(\Sigma)$ is part of a vertical cylinder over a geodesic of $\s^2$ in  $\s^2\times\R$.
\end{proposition}
\begin{proof}
Suppose that $g$ is constant. Then the normal vector $N$ is also constant on $\Sigma$, so $X(\Sigma)$ is contained in a plane $P$ of $\R^3$. If $P$ does not pass through the origin of $\R^3$, there is a $2-$sphere centered at the origin of $\R^3$ which it is tangent to $P$ at a point $z_0$. However, by the Tangency Principle, $P$ must coincide with this $2-$sphere on a neighbourhood of $z_0$, that is absurd. Therefore, $P$ passes through the origin of $\R^3$ and then $X(\Sigma)$ is part of a vertical cylinder over a geodesic of $\s^2$ in $\s^2\times\R$.

Conversely, since vertical cylinders over geodesics of $\s^2$ in $\s^2\times\R$ are planes in our model of $\s^2\times\R$, then the Gauss map $g$ must be constant.
\end{proof}

Given $X:\Sigma\to\S2R$ a minimal conformal immersion, we consider the smooth maps $p:\Sigma\to\bar{\C}$ and $r:\Sigma\to\R$ defined by $p(z)=\varphi\circ\pi_1\circ X(z)$ and $r(z)=\pi_2\circ X(z)$, respectively, that is,
\[z  \mapsto p(z)=\frac{F(z)}{|X(z)|+h(z)} \text{ \ and \ } z\mapsto r(z)=\log|X(z)|.\]

We observe that $p=0$ or $\infty$ if, and only if, $F=0$. Then $p$
cannot be identically equal to $0$ or $\infty$ on an open set, since $X$ is an immersion. Moreover, $p$ is a harmonic map into $(\bar{\C},4/(1+|w|^2)^2|\dif w|^2)$ and $r$ is a harmonic function on $\R$. Indeed, if $X$ is minimal then $\pi_j\circ X$ is a harmonic map, for $j=1,2$. Since the stereographic projection $\varphi$ is an isometry between $\s^2$ and $\bar{\C}$ endowed with the metric $4/(1+|w|^2)^2|\dif w|^2$, we have that
$\varphi\circ\pi_1\circ X=p$ and $\pi_2\circ X = r$ are harmonic maps.

\begin{proposition} Let $X:\Sigma\to\S2R$ be a minimal conformal immersion. Let $g:\Sigma\to\bar{\C}$  be the Gauss map of $X$ and suppose that $g\neq \infty$. Then the following equations hold:
  \begin{align}
     & (1+|p|^2)p_{z\bar{z}}-2\bar{p}p_zp_{\bar{z}}=0, \label{harmonicity-p}\\
     & (\bar{g}-\bar{p})^2p_z+(1+\bar{g}p)^2\bar{p}_z=0,\label{equation-p-1} \\
     & (1+|p|^2)^2g_z = (1+g\bar{p})^2p_z+(g-p)^2\bar{p}_z.\label{equation-p-2}
  \end{align}
These equations hold when $p\neq \infty$.
\end{proposition}
\begin{proof}
Indeed, equation \eqref{harmonicity-p} is the harmonicity condition for the map $p$.

We restrict ourselves on an open set of $\Sigma$ which $p\neq 0$. Away from the points where $g=0$ and $\infty$, by definition of $p$ and \eqref{eq-F}, we compute
\begin{align}
  p_z &=-\frac{\eta(|X|+h+\bar{g}F)^2}{4|X|(|X|+h)^2\bar{g}},\label{eq-pz} \\
  \bar{p}_z
  &=\frac{\eta(\bar{g}|X|+h\bar{g}-\bar{F})^2}{4|X|(|X|+h)^2\bar{g}}.\label{eq-barpz}
\end{align}

On the other hand, we have $|X|+h=F/p\in\R$,
\[\displaystyle(|X|+h+\bar{g}F)^2=\frac{F^2}{p^2}(1+\bar{g}p)^2
\text{ \ and \ }
\displaystyle(|X|+h+\bar{g}F)^2=\frac{F^2}{p^2}(1+\bar{g}p)^2.\]
Then, by equations \eqref{eq-pz} and \eqref{eq-pz2}, we get
\begin{align}
  4|X|\bar{g}p_z &= -\eta(1+\bar{g}p)^2, \label{eq-pz2}\\
  4|X|\bar{g}\bar{p}_z &= \eta(\bar{g}-\bar{p})^2. \label{eq-pbarz2}
\end{align}
Therefore, $(\bar{g}-\bar{p})^2\eqref{eq-pz2}+(1+\bar{g}p)^2\eqref{eq-pbarz2}=0$ that implies \eqref{equation-p-1}. By continuity, \eqref{equation-p-1} holds at a point where $g=0$. Analogously, equation \eqref{equation-p-1} holds on a neighbourhood of a point which $p=0$.

If $p=g$ or $p=-1/\bar{g}$ on an open set of $\Sigma$, then equation \eqref{equation-p-2} is trivial. Suppose that $g\neq 0$ on an open set of $\Sigma$. Then, away from the points where $p=g$ and $p=-1/\bar{g}$, we observe that $p_z$ is well defined by \eqref{eq-pz2} and it does not vanish. Differentiating equation \eqref{equation-p-1} with respect to $\bar{z}$ and multiplying by $(1+|p|^2)(1+\bar{g}p)$, we have
\begin{multline*}
2(1+|p|^2)(\bar{g}-\bar{p})(1+\bar{g}p)(\bar{g}_{\bar{z}}-\bar{p}_{\bar{z}})p_z +(\bar{g}-\bar{p})^2(1+\bar{g}p)\big[(1+|p|^2)p_{z\bar{z}}\big]\\
+2(1+|p|^2)(\bar{g}_{\bar{z}}p+\bar{g}p_{\bar{z}})\big[(1+\bar{g}p)^2\bar{p}_z\big]+(1+\bar{g}p)^3\big[(1+|p|^2)\overline{p_{z\bar{z}}}\big]=0.
\end{multline*}
From the equation above, using \eqref{harmonicity-p} and \eqref{equation-p-1} into to the brackets terms, we get
\[2(\bar{g}-\bar{p})p_z\Big\{(1+|p|^2)^2\bar{g}_{\bar{z}}- (1+\bar{g}p)^2\bar{p}_{\bar{z}}-(\bar{g}-\bar{p})^2p_{\bar{z}}\Big\}=0,\]
which implies \eqref{equation-p-2}. By continuity, \eqref{equation-p-2} holds at a point where $p=g$ or $p=-1/\bar{g}$. Analogously, equation \eqref{equation-p-2} holds on a neighbourhood of a point which $g=0$.
\end{proof}

\begin{remark}\label{2-spheres}
We notice that if $p=g$ (respec., $p=-1/\bar{g}=\widetilde{g}$) on $\Sigma$, by the definitions of $p$ and $g$, we have that $X/|X|=N$ (respec., $X/|X|=-N$), i.e., $X(\Sigma)$ is part of a totally geodesic $2-$sphere in $\S2R$. Moreover, in both cases, the Gauss map is holomorphic.
\end{remark}

In the next result, we derive an important second order equation satisfied by the Gauss map $g$ and the map $p$.

\begin{proposition}\label{prop-gaussmap}
The Gauss map $g$ and the map $p$ satisfy
  \begin{multline}\label{equation-gaussmap}
\big(|g-p|^2-|1+\bar{g}p|^2\big)(1+|g|^2)g_{z\bar{z}}+2(g-p)(1+g\bar{p})|g_z|^2\\
 +2\big(\bar{g}|1+\bar{g}p|^2-(\bar{g}-\bar{p})(1+|g|^2)\big)g_zg_{\bar{z}}=0
  \end{multline}
when $g\neq \infty$.
\end{proposition}

\begin{proof}
Firstly, note that if $p=g$ or $p=-1/\bar{g}$ on an open set of $\Sigma$, then \eqref{equation-gaussmap} is equivalent to \eqref{harmonicity-p}, and if $g_z=0$ or $g=0$ on an open set of $\Sigma$, then \eqref{equation-gaussmap} is trivial.

Suppose that $g\neq 0$ on an open set of $\Sigma$. Then, away from the points where $p=g$ and $p=-1/\bar{g}$, $p_z$ and $\bar{p}_z$ are well defined by \eqref{eq-pz2} and \eqref{eq-pbarz2} and they do not vanish. Differentiating equation \eqref{equation-p-2} and using \eqref{harmonicity-p}, we get
\begin{multline*}
(1+|p|^2)^2g_{z\bar{z}}+\frac{2\bar{p}p_{\bar{z}}}{1+|p|^2}\Big((1+|p|^2)^2g_z-(1+g\bar{p})^2p_z\Big) +2(g-p)p_{\bar{z}}\bar{p}_z\\
+\frac{2p\bar{p}_{\bar{z}}}{1+|p|^2}\Big((1+|p|^2)^2g_z-(g-p)^2\bar{p}_z\Big)-2g(1+g\bar{p})p_z\bar{p}_{\bar{z}} \\
-2\bar{p}(1+g\bar{p})p_zg_{\bar{z}}-2(g-p)\bar{p}_zg_{\bar{z}}=0.
\end{multline*}
By equation \eqref{equation-p-2}, we have
\begin{multline}\label{eq-prop3}
(1+|p|^2)^2g_{z\bar{z}}+\frac{2(g-p)(1+g\bar{p})}{1+|p|^2}(|\bar{p}_z|^2-|p_z|^2)\\
-2\bar{p}(1+g\bar{p})p_zg_{\bar{z}}-2(g-p)\bar{p}_zg_{\bar{z}}=0.
\end{multline}

On the other hand, by equation \eqref{equation-p-1}, we have
\begin{equation}\label{eq-pz-and-barpz}
p_z = -\frac{(1+\bar{g}p)^2}{(\bar{g}-\bar{p})^2}\bar{p}_z,
\end{equation}
and then \eqref{equation-p-2} implies
\begin{equation}\label{eq-g-z}
  (1+|p|^2)^2g_z = \frac{1}{(\bar{g}-\bar{p})^2}\big(|g-p|^4-|1+\bar{g}p|^4\big)\bar{p}_z.
\end{equation}

At a point where $g_z\neq 0$, we get $|g-p|^4-|1+\bar{g}p|^4\neq0$. We also note that $|g-p|^4-|1+\bar{g}p|^4=(1+|g|^2)(1+|p|^2)(|g-p|^2-|1+\bar{g}p|^2)$ and then we get
\begin{align}
  p_z &= - \frac{(1+|p|^2)(1+\bar{g}p)^2}{(1+|g|^2)(|g-p|^2-|1+\bar{g}p|^2)}g_z,\label{eq-pz-expression} \\
  \bar{p}_z &= \frac{(1+|p|^2)(\bar{g}-\bar{p})^2}{(1+|g|^2)(|g-p|^2-|1+\bar{g}p|^2)}g_z.\label{eq-barpz-expression}
\end{align}

Then, by \eqref{eq-prop3}, we have
\begin{multline}\label{eq-prop4}
(|g-p|^2-|1+\bar{g}p|^2)(1+|p|^2)(1+|g|^2)g_{z\bar{z}}+2(1+|p|^2)(g-p)(1+g\bar{p})|g_z|^2\\
  +2\Big(\bar{p}(1+\bar{g}p)|1+\bar{g}p|^2-(\bar{g}-\bar{p})|g-p|^2\Big)g_zg_{\bar{z}}=0,
\end{multline}
and since
\[\bar{p}(1+\bar{g}p)|1+\bar{g}p|^2-(\bar{g}-\bar{p})|g-p|^2 = (1+|p|^2)\Big(\bar{g}|1+\bar{g}p|^2-(\bar{g}-\bar{p})(1+|g|^2)\Big),\]
we get \eqref{equation-gaussmap} dividing \eqref{eq-prop4} by $(1+|p|^2)$. By continuity, \eqref{equation-p-2} holds on a neighbourhood of a point which $g_z=0$. If $p=g$ or $p=-1/\bar{g}$ at some point, again by continuity \eqref{equation-p-2} holds. Analogously, equation \eqref{equation-p-2} holds on a neighbourhood of a point which $g=0$.
\end{proof}

\begin{proposition}\label{anti-holomorphic}
There is no minimal conformal immersion $X:\Sigma\to\S2R$ which its Gauss map $g$ is an anti-holomorphic non-constant map.
\end{proposition}

\begin{proof}
By contradiction, suppose that exists a such immersion which $g_z=0$ on $\Sigma$. Away from the points where $p=g$ and $p=-1/\bar{g}$, by equation \eqref{eq-g-z}, we get that either $\bar{p}_z=0$ or $|g-p|^2-|1+\bar{g}p|^2=0$.

If $\bar{p}_z=0$ on $\Sigma$, then, by equations \eqref{equation-p-1} and \eqref{equation-p-2}, we get $p_z=0$. Thus $p$ is a constant function on $\Sigma$ and, by definition of $p$, we have $\varphi^{-1}(\text{const})=X/|X|$; however it is not possible since $X$ is an immersion.

On the other hand, by the definition of $p$, $|g-p|^2-|1+\bar{g}p|^2=0$ on $\Sigma$ if, and only if,
\begin{equation}\label{eq-omega}
F\bar{g}+\bar{F}g+h(1-|g|^2)=0.
\end{equation}
Differentiating equation \eqref{eq-omega} and using \eqref{eq-F}, we get $\bar{g}_z(F-hg)=0$, that is, $F-hg=0$ on $\Sigma$. Finally, differentiating $F-hg=0$ and using $\eqref{eq-F}$, we get $-\eta(1+|g|^2)/\bar{g}=0$ which holds if, and only if, $\eta=0$, i.e., if $g=0$ or $\infty$ on $\Sigma$.
\end{proof}

\begin{remark}\label{quociente} Suppose that $g$ is a non-constant map. By equations \eqref{equation-p-1} and \eqref{equation-p-2}, we get
\[(1+|p|^2)^3|g_z|^2 = (1+|g|^2)(|g-p|^2-|1+\bar{g}p|^2)(|\bar{p}_z|^2-|p_z|^2). \]
On the other hand, since \eqref{eq-pz2} and \eqref{eq-pbarz2}, we have
\[16|X|^2|g|^2(|\bar{p}_z|^2-|p_z|^2)=|\eta|^2(1+|g|^2)(1+|p|^2)(|g-p|^2-|1+\bar{g}p|^2),\]
then
\[16|X|^2|g|^2(1+|p|^2)^2|g_z|^2=|\eta|^2(1+|g|^2)^2(|g-p|^2-|1+\bar{g}p|^2)^2.\]
Since $g$ is not constant,  $\eta$ cannot vanish on an open set of $\Sigma$ and, by Proposition \ref{anti-holomorphic}, the same happens for $g_z$. At a point where $\eta\neq 0$ and $g_z \neq 0$, we have
\[\frac{(|g-p|^2-|1+\bar{g}p|^2)^2}{|g_z|^2}=\frac{16|X|^2|g|^2(1+|p|^2)^2}{|\eta|^2(1+|g|^2)^2}.\]
Therefore, by continuity, the left side of the equation above is a function that do not vanish on $\Sigma$, when $g$ is a non-constant map.
\end{remark}

\begin{corollary}\label{eq-inducedmetric2}
Let $X:\Sigma\to\S2R$ be a minimal conformal immersion. Then the metric induced by $X$ is given by
\[\dif s^2 =\frac{4(1+|p|^2)^2|g_z|^2}{(|g-p|^2-|1+\bar{g}p|^2)^2}|\dif z|^2\]
when $g\neq \infty.$
\end{corollary}

\begin{proof}
Let $U\subset\Sigma$ be an open set which $p\neq g$. At a point where $g\neq 0$, by equations \eqref{eq-inducedmetric1} and  \eqref{eq-pbarz2}, we have that the metric induced by $X$, in the minimal case, is given by
\[\dif s^2 = \frac{4(1+|g|^2)^2|\bar{p}_z|^2}{|g-p|^4}|\dif z|^2.\]
However, using equation \eqref{eq-barpz-expression}, we get
\begin{equation*}
\dif s^2 =\frac{4(1+|p|^2)^2|g_z|^2}{(|g-p|^2-|1+\bar{g}p|^2)^2}|\dif z|^2.
\end{equation*}
By continuity, this expression holds at a point where $g=0$. Analogously, equation \eqref{equation-p-2} holds on a neighbourhood of a point which $g=p$.
\end{proof}
%

\begin{proposition}
Let $X:\Sigma\to\S2R$ be a minimal conformal immersion with a holomorphic non-constant Gauss map $g$. Then $X(\Sigma)$ is part of a totally geodesic $2-$sphere in  $\s^2\times\R$.
\end{proposition}
\begin{proof}
If $g$ is a holomorphic non-constant function, by Proposition \ref{prop-gaussmap}, we have $(g-p)(1+g\bar{p})|g_z|^2=0$, that is, either $p=g$ or $p=-1/\bar{g}$. Therefore, $X(\Sigma)$ is part of a totally geodesic $2-$sphere in $\s^2\times\R$.
\end{proof}

\begin{lemma}\label{curvatura}
The curvature $K$ of the metric induced by $X$ is
\[K=\frac{(|g-p|^2-|1+\bar{g}p|^2)^2}{(1+|g|^2)^2(1+|p|^2)^2|g_z|^2}(|g_z|^2-|\bar{g}_z|^2).\]
This formula holds when $g\neq \infty$. In particular, $K$ vanishes at singular points of $g$.
\end{lemma}
\begin{proof}
Since $X$ is a conformal immersion, we can compute the curvature $K$ of the metric induced by $X$  through $2K=-\Delta (\log \lambda)$, i.e., $K=-(2/\lambda) (\log\lambda)_{z\bar{z}}$, where $\lambda$ is given by $\dif s^2 = \lambda |\dif z|^2$. In order to simplify our computations, we will use the expression \eqref{eq-inducedmetric1} instead of expression obtained in Corollary \ref{eq-inducedmetric2}. Then, $\lambda$ is given by
\[\lambda=\frac{(1+|g|^2)^2|\eta|^2}{4|g|^2|X|^2}.\]

We restrict ourselves to a domain on which $g\neq 0$. We observe that, since
\[(\log \lambda)_{z\bar{z}} = 2(\log (1+|g|^2))_{z\bar{z}} + (\log\eta)_{z\bar{z}} +
(\log\bar{\eta})_{z\bar{z}} -
2r_{z\bar{z}} -
(\log g)_{z\bar{z}}-(\log\bar{g})_{z\bar{z}}\]
and $r_{z\bar{z}}=0$, once $X$ is a minimal immersion, by Lemma \ref{lemma-g-eta}, we get
\begin{multline*}
(\log \lambda)_{z\bar{z}} = \frac{2}{(1+|g|^2)^2}\Big((g_{z\bar{z}}\bar{g}+|g_z|^2+|\bar{g}_z|^2+g\bar{g}_{z\bar{z}})(1+|g|^2)-(g_z\bar{g}+g\bar{g}_z)(g_{\bar{z}}\bar{g}+g\bar{g}_{\bar{z}})\Big) \\
-\frac{2\bar{g}_{\bar{z}}}{(1+|g|^2)^2\bar{g}}(g_z\bar{g}+g\bar{g}_z)
-\frac{2g_z}{(1+|g|^2)^2g}(g_z\bar{g}+g\bar{g}_z)\\
-\frac{2g}{(1+|g|^2)\bar{g}}(\bar{g}_{z\bar{z}}\bar{g}-\bar{g}_z\bar{g}_{\bar{z}})-\frac{2\bar{g}}{(1+|g|^2)g}(g_{z\bar{z}}g-g_zg_{\bar{z}}).
\end{multline*}
Follows from a direct computation that
\[(\log \lambda)_{z\bar{z}} = \frac{2}{(1+|g|^2)^2}(|\bar{g}_z|^2-|g_z|^2).\]
Therefore, the curvature $K$ is given by
\[K = \frac{16|g|^2|X|^2}{(1+|g|^2)^4|\eta|^2}(|g_z|^2-|\bar{g}_z|^2),\]
and by continuity this formula holds at a point where $g=0$.

Let $U\subset\Sigma$ be an open set on which $p\neq g$. By equation \eqref{eq-pbarz2}, we get
\[K = \frac{|g-p|^4}{(1+|g|^2)^4|\bar{p}_z|^2}(|g_z|^2-|\bar{g}_z|^2).\]
However, by equation \eqref{eq-barpz-expression}, we get
\[K=\frac{(|g-p|^2-|1+\bar{g}p|^2)^2}{(1+|g|^2)^2(1+|p|^2)^2|g_z|^2}(|g_z|^2-|\bar{g}_z|^2),\]
and by continuity, this formula holds at a point where $p=g$. Moreover, since $\det \dif g = (|g_z|^2-|\bar{g}_z|^2)$, by Remark \ref{quociente}, we have that $K$ vanishes at singular points of $g$.
\end{proof}

\begin{proposition}\label{cylinders}
Let $X:\Sigma\to\S2R$ be a minimal conformal immersion and $g$ be its Gauss map. If $g$ is a singular map then $X(\Sigma)$ is part of a vertical cylinder over a geodesic of $\s^2$ in $\s^2\times\R$. In particular, $g$ is constant.
\end{proposition}

\begin{proof}
Since $\det g= |g_z|^2-|\bar{g}_z|^2$, if the Gauss map $g$ is singular, then  $|g_z|=|\bar{g}_z|$ on $\Sigma$, and by Lemma \ref{curvatura}, we get that $K=0$. Therefore, by \cite[Torralbo-Urbano]{TorralboUrbano15}, $X(\Sigma)$ is part of a vertical cylinder over a geodesic of $\s^2$ in $\s^2\times\R$. Hence, by Proposition \ref{cylinders2}, $g$ is constant.
\end{proof}

\begin{corollary}\label{g-harmonic}
Let $X:\Sigma\to\S2R$ be a minimal conformal immersion and $g$ be its Gauss map. If $g$ is a non-constant harmonic map then $X(\Sigma)$ is part of a totally geodesic $2-$sphere in  $\s^2\times\R$.
\end{corollary}
\begin{proof}
If $g$ is a non-constant harmonic map, then $g$ satisfies
\[(1+|g|^2)g_{z\bar{z}}-2\bar{g}g_z g_{\bar{z}}=0,\]
and, by Proposition \ref{prop-gaussmap}, we get that
\[|g-p|^2|1+\bar{g}p|^2(|g_z|^2-|g_{\bar{z}}|^2)=0.\]
Since $g$ is non-constant, then $(|g_z|^2-|g_{\bar{z}}|^2)\neq 0$, otherwise, by Proposition \ref{cylinders}, $g$ must be constant. Therefore, $g=p$ or $g=-1/\bar{p}$ and then $X(\Sigma)$ is part of a totally geodesic $2-$sphere in $\s^2\times\R$ (Remark \ref{2-spheres}).
\end{proof}

\begin{lemma}\label{lemma-r} Let $X:\Sigma\to\S2R$ be a minimal conformal immersion. Then the Gauss map $g$ and the maps $p$ and $r$ satisfy
\begin{equation}\label{eq-r}
r_{z} = \frac{2(\bar{g}-\bar{p})(1+g\bar{p})}{(1+|g|^2)(|g-p|^2-|1+\bar{g}p|^2)}g_z
\end{equation}
when $g\neq \infty$.
\end{lemma}
\begin{proof}
Let $U\subset\Sigma$ be an open set on which $p\neq g$. Then, by equation \eqref{eq-pz-and-barpz}, $\bar{p}_z$ does not vanish on $U$, otherwise $p$ is a constant function. By Lemma \ref{lemma-g-eta} and equation \eqref{eq-pbarz2}, we get
\[\frac{1-|g|^2}{1+|g|^2}(\log \bar{g})_{\bar{z}} = \Bigg(\log \frac{4|X|\bar{g}\bar{p}_z}{(\bar{g}-\bar{p})^2} \Bigg)_{\bar{z}}.\]

Away from the points where $g=0$ and $\infty$, we have
\[-\frac{2|g|^2}{1+|g|^2}\frac{\bar{g}_{\bar{z}}}{\bar{g}}
=r_{\bar{z}}+\frac{\bar{p}_{z\bar{z}}}{\bar{p}_z}-2\frac{\bar{g}_{\bar{z}}-\bar{p}_{\bar{z}}}{\bar{g}-\bar{p}}
\]
and, by equation \eqref{harmonicity-p}, we obtain
\[r_{\bar{z}}=\frac{2(1+g\bar{p})}{(\bar{g}-\bar{p})(1+|g|^2)}\bar{g}_{\bar{z}}-\frac{2(1+\bar{g}p)}{(\bar{g}-\bar{p})(1+|p|^2)}\bar{p}_{\bar{z}}.
\]
By continuity, this holds on a neighbourhood of a point which $g=0$. Then, using equation \eqref{eq-pz-expression}, we get the conjugate of our assertion and, by continuity, it also holds at a point where $p=g$.
\end{proof}

In the following we establish our main result. As was seen before in the Proposition \ref{cylinders},
vertical cylinders over geodesics of $\s^2$ in $\S2R$ are the only surfaces which the Gauss map is singular. Moreover, any two of them differ by an isometry of $\S2R$. The next result prove that the same rigidity happens for any two minimal conformal immersions into $\S2R$ with the same non-constant Gauss map.

\begin{theorem}\label{th-A}
Let $X:\Sigma\to\S2R$ be a minimal conformal immersion and $g$ be its non-constant Gauss map. If $\hat{X}:\Sigma\to\S2R$ is another minimal conformal immersion with the same Gauss map of $X$, then $\hat{X}=f\circ X$, with $f\in\iso{\S2R}$ given either by $f=(\id,T)$ or $f=(\mathcal{A},T)$, where $\mathcal{A}$ denotes the antipodal map on $\s^2$ and $T$ is a translation on $\R$.
\end{theorem}

\begin{proof}
Firstly, we can assume that $g$ is a regular map. Otherwise, by Proposition \ref{cylinders}, we get that $X(\Sigma)$ is a vertical cylinder over a geodesic of $\s^2$ in $\S2R$, and then $g$ is constant.

Suppose that $g$ is a regular map on $\Sigma$. Then $|g|$ cannot be constant on an open set $U\subset\Sigma$, otherwise $g_z\bar{g}=-g\bar{g}_z$ and follows that $g$ is a singular map. So, we restrict ourselves to this open set $U\subset\Sigma$ on which $|g|\neq 1$.

By Proposition \ref{prop-gaussmap}, for the minimal conformal immersion $X:\Sigma\to\S2R$, we get
\begin{equation}\label{eq-gaussmap-2}
  A(1-|p|^2)+Bp+C\bar{p}=0,
\end{equation}
where the coefficients $A, B$ and $C$ depend on $g,\bar{g}$ and theirs derivatives, given by
\begin{align*}
  A &= -(1-|g|^4)g_{z\bar{z}}+2g|g_z|^2-2|g|^2\bar{g}g_zg_{\bar{z}}, \\
  B &= -2\bar{g}(1+|g|^2)g_{z\bar{z}}-2|g_z|^2+2\bar{g}^2g_zg_{\bar{z}},  \\
  C &= -2g(1+|g|^2)g_{z\bar{z}}+2g^2|g_z|^2+2(1+2|g|^2)g_zg_{\bar{z}}.
\end{align*}

If $A=0$ at some $z_0\in U$, at this point we have $(1-|g|^2)Bp+(1-|g|^2)C\bar{p}=0$ and then
\[2(1+|g|^2)g_z\Big\{(\bar{g}^2g_{\bar{z}}-\bar{g}_{\bar{z}})p+(g_{\bar{z}}-g^2\bar{g}_{\bar{z}})\bar{p}\Big\}=0.\]
When $g_z(z_0)\neq 0$, the equation above has unique solution $p(z_0)=0$ if, and only if, $|\bar{g}^2g_{\bar{z}}-\bar{g}_{\bar{z}}|^2-|g_{\bar{z}}-g^2\bar{g}_{\bar{z}}|^2\neq 0,$ that is, $(1-|g|^4)(|g_z|^2-|\bar{g}_z|^2)\neq 0$ which holds in $U$. Moreover, $A$ cannot vanish identically on an open set of $\Sigma$, since this would imply that $X$ is not an immersion.

If $A\neq 0$  at $z_0\in U$, then we consider the system
\begin{align}
  A(1-|p|^2)+Bp+C\bar{p}&= 0, \label{eq-1}\\
  \bar{A}(1-|p|^2)+\bar{C}p+\bar{B}\bar{p}&=0.\label{eq-2}
\end{align}
Since $\bar{A}\eqref{eq-1}+A\eqref{eq-2}$ implies
\begin{equation}\label{eq-6}
(\bar{A}B-A\bar{C})p+(\bar{A}C-A\bar{B})\bar{p}=0
\end{equation}
then, multiplying \eqref{eq-1} by $(\bar{A}C-A\bar{B})$ and dividing by $A$, we get
\begin{equation}\label{eq-3}
(\bar{A}B-A\bar{C})p^2+(|C|^2-|B|^2)p+(\bar{A}C-A\bar{B})=0,
\end{equation}
i.e., $\alpha p^2+\beta p-\bar{\alpha}=0$, with $\alpha=\bar{A}B-A\bar{C}$ and $\beta =|C|^2-|B|^2 $. Considering that the discriminant of the equation above is $\beta^2+4|\alpha|^2$, then equation \eqref{eq-3} has two distinct solutions for $p(z_0)$ if, and only if, $\alpha\neq 0$, that is, $\bar{A}B-A\bar{C}\neq 0$ (once  $\bar{A}B=A\bar{C}$ implies $|B|^2=|C|^2$).

On the other hand, note that
\begin{align}
2\bar{g}A-(1-|g|^2)B &= 2(1+|g|^2)g_zD,\label{eq-4}\\
2gA-(1-|g|^2)C &= 2(1+|g|^2)g_z E,\label{eq-5}
\end{align}
where $D=\bar{g}_{\bar{z}}-\bar{g}^2g_{\bar{z}}$ and $E=g^2\bar{g}_{\bar{z}}-g_{\bar{z}}$. Then $\bar{A}\eqref{eq-4}-A\overline{\eqref{eq-5}}$ implies that
\[(1-|g|^2)(A\bar{C}-\bar{A}B)=2(1+|g|^2)(g_z\bar{A}D-\bar{g}_{\bar{z}}A\bar{E}).\]
When $g_z(z_0)\neq 0$, then $\alpha=0$ implies $|D|^2=|E|^2$. By definition of $D$ and $E$, $|D|^2=|E|^2$ if, and only if, $(1-|g|^4)(|g_z|^2-|\bar{g}_z|^2)= 0$ which does not happen on $U$. Therefore, equation \eqref{eq-3} has two distinct solutions for each $p(z_0)$ on $U$, when $g$ is a regular map.

Let $\hat{X}:\Sigma\to\S2R$ be a minimal conformal immersion with the same regular Gauss map $g$ of $X$. In the following, we use the symbol  $\hat{ \ }$ to refer to objects related to $\hat{X}$.

If $p=\hat{p}$ then, by definition, we have that $X/|X|=\hat{X}/|\hat{X}|$, i.e., $\hat{X}=q(z)X$, where $q:\Sigma\to\R^+$ is a smooth function, given by $q(z)=|\hat{X}(z)|/|X(z)|$. However, we observe that the function $q$ is constant.
In fact, by Lemma \ref{lemma-r}, we get that $(\log q)_{\bar{z}}=0$, since $X$ and $\widehat{X}$ have the same Gauss map. Then $q$ is a real holomorphic function, that is, $q(z)=q_0$ is constant and therefore $\hat{X}=f\circ X$, where $f=(\id,T)\in\iso{\S2R}$ is a vertical translation of $\S2R$ with $T(t)=t+q_0$.

If $p\neq \hat{p}$ then we claim that $\hat{p}=-1/\bar{p}$. Indeed, by a direct computation, we get that $-1/\bar{p}$ satisfies equation \eqref{equation-gaussmap}. Since \eqref{eq-3} has exactly two distinct solutions when $g$ is a regular map, then $\hat{p}=-1/\bar{p}$. Moreover, $\hat{p}\bar{p}=-1$ implies that $\varphi^{-1}(\hat{p})$ and $\varphi^{-1}(p)$ are antipodal points on $\s^2$, that is, $\hat{X}=q(z)\mathcal{A}\circ X$, where $\mathcal{A}$ is the antipodal map on $\s^2$ and $q(z)=|\hat{X}(z)|/|X(z)|$. However, by Lemma \ref{lemma-r}, we get $(\log q)_{\bar{z}}=0$, since $X$ and $\hat{X}$ have the same Gauss map and $p$ and $\hat{p}$ are antipodal points. Then $q$ is a real holomorphic function, that is, $q(z)=q_0$ is constant and therefore $\hat{X}=f\circ X$, where $f=(\mathcal{A},T)\in\iso{\S2R}$ with $T(t)=t+q_0$ a translation on $\R$.
\end{proof}

\begin{remark}
We notice that the minimal conformal immersion $X:\Sigma\to\S2R$ can be recovered from of the Gauss map $g$, up to isometries described on Theorem \ref{th-A},
when $g$ is a non-constant map. In fact, by definition of $p$, we have $X=|X|\varphi^{-1}(p)$. On the other hand, by Lemma \ref{eq-r}, $|X|$ can be recovered from of $g$, up to a multiplicative constant, and the map $p$ can be recovered from of $g$ as the solution of equation \eqref{eq-3}, up to change of the antipodal map on $\s^2$.
\end{remark}

\begin{remark}[Chirstoffel's problem] In 1867, E. B. Christoffel \cite{Christoffel1867} considered the following problem: Given an immersion $X:\Sigma\to \R^3$ of an oriented Riemannian surface $\Sigma$, when does $X$ possess a non-congruent immersion $\hat{X}:\Sigma\to\R^3$, such that the tangent planes at each point of $X$ and $\hat{X}$ are parallels and the metrics induced by $X$ and $\hat{X}$ are conformally equivalents? In the literature, this problem is known as Christoffel's problem (we may refer to \cite{DajczerTojeiro16}, \cite[Chapter 9]{Jensen16} and \cite{Samuel47}).

The local solution found by Chirstoffel depends of the orientation of $X$ and $\hat{X}$. Suppose that there are $X$ and $\hat{X}$ satisfying the Chirstoffel's problem. If $X$ and $\hat{X}$ induce the same orientation at each point of $\Sigma$, then $X$ is minimal and $\hat{X}$ is conformal associate of $X$. Otherwise, if $X$ and $\hat{X}$ induce opposite orientations at each point of $\Sigma$, then $X$ is isothermic and, in this case, the immersion $\hat{X}$ is called \emph{Christoffel transform} of $X$.

Since we identify $\S2R$ with $\R^3\setminus \{0\}$, endowed 
with a metric conformally equivalent to the Euclidean metric of $\R^3$, a conformal immersion $X:\Sigma\to\S2R$ is a conformal immersion $X:\Sigma\to\R^3$ as well. Moreover, the definition of Gauss map in $\S2R$ coincides with the definition in $\R^3$. Then given $X$ and $\hat{X}$, the condition $g=\hat{g}$ implies that the tangent planes are parallels at each point of $\Sigma$ and $X$ and $\hat{X}$ induce the same orientation. By the Christoffel's Theorem, either $X$ and $\hat{X}$ differ by an isometry of $\R^3$ or $X$ and $\hat{X}$ are both minimal conformal immersions on $\R^3$. In the second case, we get that $g$ is anti-holomorphic and, by Proposition \ref{anti-holomorphic}, $g$ must be a constant map. Then $X(\Sigma)$ and $\hat{X}(\Sigma)$ are parts of planes as surfaces in $\R^3$ with the same constant Gauss map. Therefore, $X(\Sigma)$ and $\hat{X}(\Sigma)$ are parts of the same vertical cylinder over a geodesic of $\s^2$ in  $\s^2\times\R$.
\end{remark}

\section{Examples}

As previously shown in the last section, minimal conformal immersions with constant Gauss map are vertical cylinders over geodesic of $\s^2$ in $\S2R$. In the case where the Gauss map $g$ is a non-constant holomorphic map, we have totally geodesic $2-$spheres in $\S2R$. These ones are the simplest examples of minimal surfaces in $\S2R$.

Next, we describe briefly classical examples studied by Pedrosa and Ritoré \cite{PedrosaRitore99}, Rosenberg \cite{rosenberg2002} and Daniel \cite{danieltams09} using our model of $\S2R$.

\begin{example}[Helicoids]  Let $\beta \neq 0$. We consider the conformal immersion
\[X_{\beta}(u+iv) = \big(e^v\sin\rho(u)e^{i\beta v},e^v \cos\rho(u)\big),\]
where the function $\rho$ satisfies
\begin{equation}\label{edo-helicoid}
\rho'^2(u)=1+\beta^2\sin^2\rho(u).
\end{equation}

This conformal immersion was presented in \cite[Section 4.2]{danieltams09} and corresponds to the minimal helicoid in $\S2R$. Following Daniel's terminology, for $\beta>0$ we say that $X_{\beta}$ is a right helicoid and for $\beta<0$ we say that $X_{\beta}$ is a left helicoid.

We can assume that $\rho'(u)>0$. We compute the normal vector $N_{\beta}$ by
\[N_{\beta}(u+iv) = \frac{1}{\rho'(u)}\begin{bmatrix}
                                            \sin\beta v+\beta\sin^2\rho(u)\cos\beta v
                                             \\
                                            -\cos\beta v+\beta\sin^2\rho(u)\sin\beta v \\
                                            \beta\sin\rho(u)\cos\rho(u) \\
                                          \end{bmatrix}.\]
Then the Gauss map $g_{\beta}$ is given by
\[g_{\beta}(u+iv) = \frac{(\beta\sin^2\rho(u)-i)e^{i\beta v}}{\rho'(u)+\beta\sin\rho(u)\cos\rho(u)}.\]

Moreover, the $r_{\beta}$ is given by $r_{\beta}(u+iv)=v$ and $p_{\beta}$ is given by
\[p_{\beta}(u+iv)=\frac{\sin\rho(u)}{1+\cos\rho(u)}e^{i\beta v}.\]

We note that when $\beta\to 0$ then $g_0(u+iv) = -i$, that is, $X_0$ is a vertical cylinder over a geodesic of $\s^2$ in  $\s^2\times\R$.

In order to obtain an explicit parametrization of the minimal helicoid in $\S2R$, we consider the following change of coordinates
\[u(s)=\int_{0}^{s}\frac{1}{\sqrt{1+\beta^2\sin^2(\sigma)}}\dif\sigma,\]
and we note that $\rho\circ u(s)=s$ satisfies equation \eqref{edo-helicoid}. In terms of these coordinates, we consider the immersion $\widetilde{X}_{\beta}:\C\to\S2R$ given by
\[\widetilde{X}_{\beta}(s+iv) = \big(e^v\sin(s)e^{i\beta v},e^v \cos(s)\big),\]
and, therefore, the surface $\widetilde{X}_{\beta}(\C)$ is the analytic continuation of the minimal helicoid in $\S2R$.

\begin{figure}[!ht]
\centering
\captionsetup{justification=centering}
\caption{\bf Part of the minimal helicoid in $\S2R$, with $\beta = 4$:}
\subfigure{\includegraphics[width=70mm]{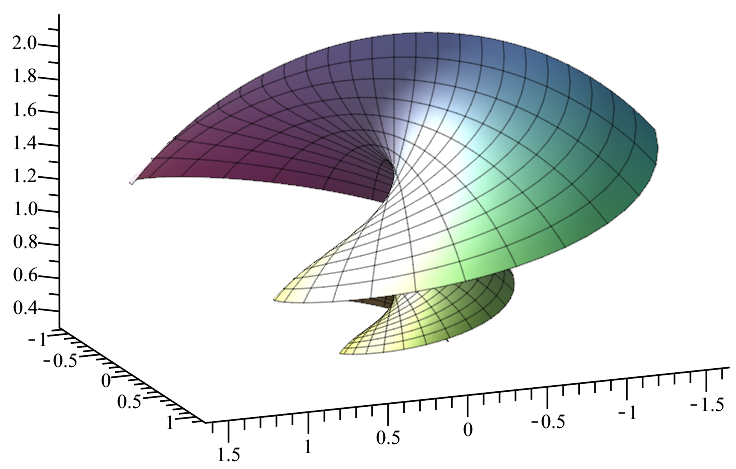}}\vfil
\caption{\bf Minimal helicoid in $\S2R$, with $\beta = 4$:}
\subfigure{\includegraphics[width=90mm]{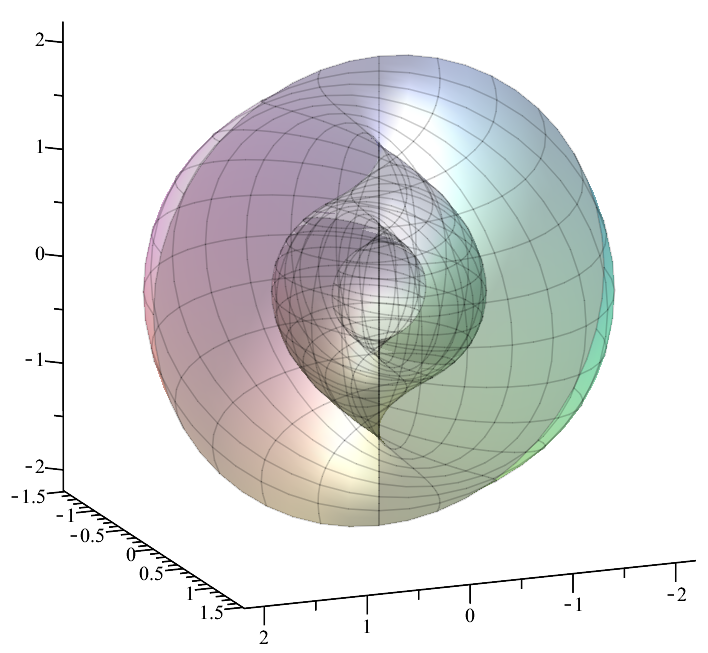}}
\end{figure}
\end{example}
\newpage

\begin{example}[Unduloids] Let $\alpha\in\R\setminus[-1,1]$. We consider the conformal immersion
\[X_{\alpha}(u+iv) = \big(e^u\sin\rho(u)e^{i\alpha v},e^u \cos\rho(u)\big),\]
where the function $\rho$ satisfies
\begin{equation}\label{edo-unduloid}
\rho'^2(u)+1=\alpha^2\sin^2\rho(u).
\end{equation}

This conformal immersion was presented in \cite[Section 4.2]{danieltams09} and corresponds to the minimal unduloid in $\S2R$.

We can assume that $\rho'(u)>0$ and $\rho(u)\in (0,\pi)$. We compute the normal vector $N_{\alpha}$ by
\[N_{\alpha}(u+iv) = \frac{1}{\alpha\sin\rho(u)}\begin{bmatrix}
                                            (\rho'(u)\sin\rho(u)-\cos\rho(u))\cos\alpha v
                                             \\
                                (\rho'(u)\sin\rho(u)-\cos\rho(u))\sin\alpha v \\
                            \sin\rho(u)+\rho'(u)\cos\rho(u) \\
                                          \end{bmatrix}.\]
Then the Gauss map $g_{\alpha}$ is given by
\[g_{\alpha}(u+iv) = \frac{(\rho'(u)\sin\rho(u)-\cos\rho(u))e^{i\alpha v}}{(1+\alpha)\sin\rho(u)+\rho'(u)\cos\rho(u)}.\]

Moreover, the $r_{\alpha}$ is given by $r_{\alpha}(u+iv)=u$ and $p_{\alpha}$ is given by
\[p_{\alpha}(u+iv)=\frac{\sin\rho(u)}{1+\cos\rho(u)}e^{i\alpha v}.\]

We note that when $\alpha\to \pm 1$ then $\rho'^2(u)+\cos^2\rho(u)=0$, that is, $\rho(u)=\pi/2$. Since
\[N_{-1}=\begin{bmatrix}
           0 \\
           0 \\
           -1
         \end{bmatrix} \text{ \ and \ }
         N_{+1}=\begin{bmatrix}
           0 \\
           0 \\
           1
         \end{bmatrix},\]
we have that $g_{-1}(u+iv)=\infty$ and $g_{+1}(u+iv)=0$. Therefore, $X_{-1}$ and $X_{+1}$ are vertical cylinders over geodesics of $\s^2$ in $\s^2\times\R$.

In order to obtain an explicit parametrization of the minimal unduloid in $\S2R$, we consider the following change of coordinates
\[u(s)=\int_{0}^{s}\frac{1}{\sqrt{1+(\alpha^2-1)\sin^2(\sigma)}}\dif\sigma,\]
and we note that
\[\rho\circ u(s)
=\arcsin\Bigg(\frac{\sqrt{1+(\alpha^2-1)\sin^2(s)}}{\alpha}\Bigg)\]
satisfies equation \eqref{edo-unduloid}. In terms of these coordinates, we consider the immersion $\widetilde{X}_{\alpha}:\C\to\S2R$ given by
\[\widetilde{X}_{\alpha}(s+iv) = \frac{e^{u(s)}}{\alpha}\Big(\sqrt{1+(\alpha^2-1)\sin^2(s)} e^{i\alpha v},\sqrt{\alpha^2-1}\cos(s)\Big),\]
and, therefore, the surface $\widetilde{X}_{\alpha}(\C)$ is the analytic continuation of the minimal unduloid in $\S2R$.

\begin{figure}[!ht]
\centering
\captionsetup{justification=centering}
\caption{\bf  Generating curve of the minimal unduloid in $\S2R$, with $\alpha = 8$:}
\subfigure{\includegraphics[width=55mm]{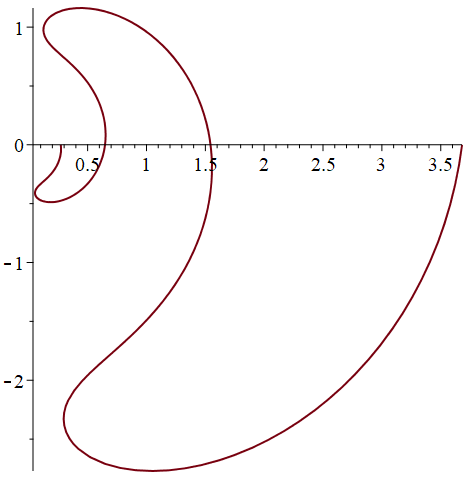}}\vfil
\caption{\bf Minimal unduloid in $\S2R$, with $\alpha = 8$:}
\subfigure{\includegraphics[width=110mm]{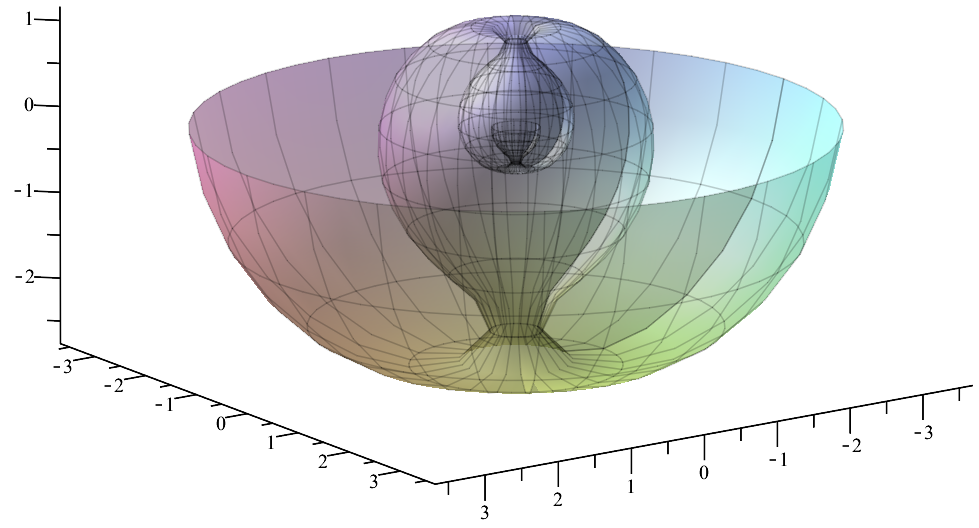}}
\end{figure}
\end{example}
\newpage

\bibliographystyle{amsplain}
\bibliography{references}

\end{document}